\documentclass[11pt,a4paper]{amsart}
\usepackage{amsfonts}
\usepackage{amsthm}
\usepackage{amsmath}
\usepackage{amscd}
\usepackage[latin2]{inputenc}
\usepackage{t1enc}
\usepackage[mathscr]{eucal}
\usepackage{indentfirst}
\usepackage{graphicx}
\usepackage{graphics}
\usepackage{pict2e}
\usepackage{epic}
\usepackage{amsmath}
\usepackage{amsfonts}
\usepackage[all,cmtip]{xy}
\usepackage{amsmath, amssymb}
\usepackage{amsmath,amssymb,tikz-cd} 
\numberwithin{equation}{section}
\usepackage[margin=2.9cm]{geometry}
\usepackage{epstopdf}

\theoremstyle{plain}
\newtheorem{Th}{Theorem}[section]
\newtheorem{Lemma}[Th]{Lemma}
\newtheorem{Cor}[Th]{Corollary}
\newtheorem{Prop}[Th]{Proposition}

 \theoremstyle{definition}

\newtheorem{Rem}[Th]{Remark}
\newtheorem{?}[Th]{Problem}

\begin{document}

\title{\bf {On the Steinberg character of a reductive p-adic group}}
\author[K. BETTAIEB]{KAREM BETTAIEB*}
\author[I. HICHRI]{IMED HICHRI**} 
\begin{abstract} The aim of this paper is to give a generalization of the construction of the Steinberg tempered character on a connected reductive p-adic group. We prove that this character is invariant by the weak restriction of the Jacquet module by analogy to finite reductive groups.
\end{abstract}

\maketitle
\hspace{0.8cm}\small{Keywords. Parabolic induction, duality, weak constant term, character of Steinberg.}\\

\hspace{0.8cm}\textsl{\bf{2010 Mathematics Subject Classification:}} \textsl{11E95, 20G05, 20G15, 22D12.}
\maketitle
\vskip 2em
\let\thefootnote\relax\footnotetext{\\
*Taief University, Kindom of Arabia Saudia.\\
E.mail: Karem@tu.edu.sa\\
**Facult\'e des sciences de Sfax, d\'epartement de math\'ematiques,
route de Soukra, 3038 Sfax, Tunisie.\\
E.mail: imedhichri69@gmail.com}
\maketitle
\vskip 2em

\large{\section{Introduction}} 
Let G be a connected, reductive p-adic group, and $\mathcal{V}(G)$ the set of tempered virtual characters of $G$ [1], [14], that is the set of finite linear combinations of characters of irreducibles tempered representations of G. In [6], K. Betta\"ieb has defined an involution on $\mathcal{V}(G),$ denoted by $D^{t}_{G}$ similar to the Curtis-Alvis duality for the characters of finite reductive groups [10], and to A.M.\hspace{0,1cm}Aubert [3] in the Grothendieck group of the category of smooth finite length representations
of reductive p-adic group. If $M$ is a Levi subgroup of $G,$ [6] this involution commutes on the one hand with the induction functor,  $ i_{G, M}: \mathcal{V}(M)\to\mathcal{V}(G)$, and on the other hand with the Jacquet's  weak restriction functor, $r^{t}_{M,G}: \mathcal{V}(G) \to \mathcal{V}(M)$ [14].\\

Let $P = MN$ be a standard parabolic subgroup of G, and  $\sigma\in \Pi_{2}(M),$ the set of equivalence class of discrete series representation of M. We denote by $i_{G,M}(\sigma)$ the set of equivalence
class of the tempered representation of G parabolically induced from $\sigma$, $\Pi_{\sigma}(G)$ the set of all its irreducible constituents, and $R_\sigma:= R^G_\sigma$ its corresponding Knapp-Stein R-group [15]. It is a finite groupe defined in terms of Plancherel measures, and is the key of the determination of the intertwining algebra $\mathcal{C}(\sigma)$  which is isomorphic to $\mathbb{C}[R_\sigma]_\eta$ the complex algebra of $R_\sigma$, with multiplication twisted by a cocycle $\eta$ . Let $1\to Z_{\sigma}\to \widetilde{R_\sigma} \to R_{\sigma}\to 1$, be a fixed central extension of $R_\sigma$, on which the cocycle associated to the intertwining operators splits. \hspace{0,1cm}J.\hspace{0,1cm}Arthur\hspace{0,1cm}[1], showed that there is a bijection: $\rho \longleftrightarrow \pi_\rho $ between $\Pi_\sigma(G)$ and $\Pi(\widetilde{R_\sigma}, \chi_{\sigma})$, the set of equivalence classes of irreducible
representations of $\widetilde{R_\sigma}$ with $ \chi_\sigma$ as $ Z_{\sigma}$-central character. Given any $\rho\in\Pi(\widetilde{R_\sigma}, \chi_{\sigma})$, we denote by $\theta_\rho$ its character and by $\Theta_{\pi_{\rho}}$ the character of $\pi_\rho$, then this bijection induces an isomorphism  $I^G$ between the respective characters of these representations, i.e, $I^{G}(\theta_{\rho}) = \Theta_{\pi_{\rho}}$, we say that $\Theta_{\pi_{\rho}}$ \textbf{corresponds} to $\theta_{\rho},$ and that $D^t_G(\Theta_{\pi_{\rho}}),$ is the dual of $\Theta_{\pi_{\rho}}$.\\

Similar to the duality $D^t_G$ on $\mathcal{V}(G)$, we establish another duality on the space $\mathbb{C}(\Pi(\widetilde{R_\sigma}, \chi_{\sigma}))$ of characters of irreductibles representations in $\Pi(\widetilde{R_\sigma}, \chi_{\sigma})$. We prove that the sign application $\xi_{\widetilde{R_\sigma}} $ of $\widetilde{R_\sigma}$, defined by $\xi_{\widetilde{R_\sigma}}(\tilde{r}):= (-1)^{\dim(A_{L_r})},$ where $\tilde{r}\in\widetilde{R_\sigma}$ is the inverse image of $r\in R_\sigma$, and $L_r$ is a Levi subgroup of G  which verifies  certain condition, is the character of an irreducible representation in $ \Pi (\widetilde{R_\sigma}, \chi_{\sigma}), $.\\

The second result of this paper is to define a special character in  $\mathcal {V}(G)$, similar to that of Steinberg,(c.f. [12], [9, \S 4.7]), which will be called \textbf{the Steinberg character associated to a pair $(M, \sigma)$}, with $M$ is a Levi subgroup of $G$ and $\sigma\in\Pi_2(M)$ (such pairs are called discrete pairs), and denoted by $St_{(M,\sigma)}^{G}$. We prove  that it is an irreducible character, that it corresponds to $\xi_{\widetilde{R_\sigma}},$ elliptic if and only if the group $ R_{\sigma} $ is elliptic, and that its dual is $\Theta_{\pi_{1}},$ where $\pi_{1}\in \Pi_{\sigma}(G)$ corresponds to the trivial representation. In addition if $L$ is a Levi subgroup of $G$, containing $M$ and satisfies certain condition then: $$r^{t}_{L,G}(\sigma)(St_{(M,\sigma)}^{G}) = St_{(M,\sigma)}^{L}$$ 
where $r^{t}_{L,G}(\sigma)$ is the projection of $r^{t}_{L,G}$ on $\Pi_\sigma(L)$.     
\maketitle
\vskip 2em

\large{\section{Preliminaires}}
Let F be a locally compact, non-discrete, nonarchimedean field of characteristic
zero, and $\mathbf{G}$ is a connected reductive algebraic
group over $F$. Let $G:= \mathbf{G} (F)$ be the group of $F$-rational points of $ \mathbf{G}$, and $A_G$ the split component of $G$, i.e, the maximal F-split torus lying in the center of $G$. We fix a Levi subgroup  $M_{0}$ of a certain minimum parabolic subgroup $ P_ {0} $ of G defined over F, and $A_0$ its split component. Any parabolic subgroup $P$ of $G$ which is defined over F, has a unique Levi decomposition, $P=M_P\ltimes N_P= M_P.N_P$, where $M_P$ is connected, reductive subgroup of $G$, and $N_P$ is the unipotent radical of $P$, if P contains $P_0$, we say that it is standard. The Levi subgroups of G, are the centralizers in $ G $ of their splits components, therefore they are uniquely determined by these components and vice versa.\\

 Let $\Phi(G, A_0)$ be the set of reduced roots of $A_0$ in the Lie algebra of G, and $\Delta\subset\Phi^{+}(G, A_0)$ the collection of simple roots, then the conjugacy classes of standards parabolics subgroups of
$G$ are in one to one correspondence with subsets of $\Delta$, [2, Lemma $1.2$], written as follows:$$I\subset\Delta\longleftrightarrow P_I= M_I N_I$$
such that: $$A_{M_I}= A_I= (\bigcap_{\alpha\in I} \ker\alpha \cap A_0)^\circ,~\textsl{and}~ M_I = Z_G(A_I)$$\\
where $Z_G(A_I)$ is the centralizer of $A_I$ in $G.$ We refer the reader to [8] for more details. We write $\mathcal{L}$ for the finite set of Levi subgroups of G. Given any $M\in\mathcal{L}$, with split component $A_M$, we write $\mathcal{L}(M):= \mathcal{L}^G(M)$ for the set of Levi subgroups of $G$ which contain $M$, and $\mathcal{P}(M)$ for the set of parabolic subgroups $P$, with Levi component $M_P= M$, any $P\in\mathcal{P}(M)$ admits a unique opposite parabolic subgroup $\overline{P}= M\overline{N}_P$, such that $P\cap\overline{P}= M$\\

For a fixed $M\in\mathcal{L}$, let $X(M)_F$ be the group of F-rational characters of $M$, and $$\mathfrak{a}_M= \textsl{Hom}( X(M)_F, \mathbb{R})$$ $$\mathfrak{a}^*_M= X(M)_F\otimes_{\mathbb{Z}} \mathbb{R}\hspace{0.7cm}$$
then $\mathfrak{a}_M$ is the real Lie algebra of $A_M$, and $\mathfrak{a}^*_M$ is its dual. We denote by $ W_0:=W_0^G:= N_G(A_0)/M_0$ the Weyl group of $G$ with respect to $A_0$, and $W:= W^G$ for the Weyl group of $G$ with respect to $A_M$, where $N_G(A_0)$ is the normalizer of $A_0$ in $G$. \\ 

An element $ x \in G $ is called  regular if $ D_{G} (x) \not = 0 $, where $ \hspace {0,2cm} D_{G} $ is the standard discriminant factor defined in [12], we denote by $G_{\textsl{reg}}$ the set of regulars elements of $G$. An element of $G_{\textsl{reg}}$ is elliptic if its centralizer is compact modulo $ A_ {G} $, we write $ G_ {\textsl{ell}} $ for the set of regular elliptic elements of $G$. 
Let $ \Pi (G)$ be the set of  equivalence classes of irreducible tempered representation of $G$, and $ \Pi_ {2} (G) $ the
subset of $ \Pi (G)$ consisting of discretes series representations of $G$. If there is no confusion, we do not make a
distinction between each equivalence class and its representative. Given any $\pi\in\Pi(G)$, we write $ \Theta _ {\pi} $ for the character of $ \pi$ which is locally integrable function on $G$ [12], and $ \Theta _ {\pi}^ {e} $ for its restriction to
$G_ {\textsl{ell}}$. The representations $\pi$ is said to be elliptic, if $ \Theta _ {\pi}^ {e} \neq 0$\\
\subsection{The weak constant term}  
Let $(\pi, V)$ be an admissible representation of $G$ and $P\in\mathcal{P}(M)$. For any quasi-character $\chi$ of $A_G$, we denote by:
$$V_\chi= \lbrace v\in V;~ \exists d\in\mathbb{N},~\forall a\in A_G, (\pi(a)- \chi(a).Id)^{d} v = 0\rbrace$$
and $$ \textsl{Exp}(\pi)=\lbrace \chi;~ V_\chi\neq 0\rbrace$$
$\textsl{Exp}(\pi)$ is called the set of exponent of $\pi$, we get then: $$V= \bigoplus_{\chi\in\textsl{Exp}(\pi)} V_\chi.$$ let $\overline{P}= M\overline{N}$ the opposite parabolic subgroup of P, and $V_{\overline{N}}:= V/{ V(\overline{N})}$ where $V(\overline{N})$ is the $M-$invariant subspace of $V$ genereted by $\lbrace \pi(\overline{n})v- v;~ \overline{n}\in\overline{N}, v\in V\rbrace$. For all $m\in M$ and $v\in V$, we define:
$$ \pi_{\overline{N}}(m) p(v)= {\delta_{\overline{P}}(m)}^{-1/2} p (\pi(m)v)$$
with $p: V\longrightarrow V_{\overline{N}}$ is the canonical projection and $\delta_{\overline{P}}$ is the modular function of $\overline{P}.$ The representation $(\pi_{\overline{N}}, V_{\overline{N}})$ is called the normalized Jacquet module of $(\pi, V)$ corresponding to $\overline{P}$, which is admissible and of finite length $[17, \S 2.3]$. We write $(\Theta_\pi)_M:= \Theta_{\pi_{\overline{N}}}$, for the character of $\pi_{\overline{N}}$, and we call it
the constant term  of $\Theta_\pi$ along $\overline{P}$, if $\pi$ is tempered, we write $(\Theta_\pi)^{w}_M:= \Theta_{\pi^{w}_{\overline{N}}}$, for the character of $\pi^{w}_{\overline{N}}$, which is the maximum tempered quotient of the normalized Jacquet module $ \pi _ {\overline{N}} $ corresponding to $ \overline {P}$, defined by:$$V^{w}_{\overline{N}}= \bigoplus_{\chi\in\textsl{Exp}(\pi_{\overline{N}})\cap \widehat{A}}
V_{\overline{N}, \chi}$$
we call $(\Theta_\pi)^{w}_M$, the weak constant term of $\Theta_\pi$ along $\overline{P}$.\\
\subsection{The set \textsl{\bf{$\mathcal{V}(G)$}}}
 A character $ \Theta $ of $G$, is called virtual, if there are a finite number of $ \pi_{1}, ..., \pi_{k} \in \Pi (G) $ and $ c_ {i} \in {\mathbb {C}}$, for all $1\leq i\leq k$, such that:
$$
\Theta = \sum_{1 \leq i \leq k} c_ {i} \Theta _ {\pi_ {i}}.
$$
We denote by $\mathcal{V}(G)$ the set of such characters, which is the set of finite linear combinations of characters of irreducible tempered representations of $G$, we say that a character in $\mathcal{V}(G)$ is irreductible if it is the character of an irreductible representation in $\Pi(G)$. Let $\mathcal{V}_{st}(G)$ be the subset of $\mathcal{V}(G)$ formed by so-called supertempered characters [13], and defined according to [14] by:
\begin{equation}
\mathcal{V}_{st}(G)=\lbrace \Theta\in\mathcal{V}(G);~\Theta^w_L= 0, \forall L\in\mathcal{L}, L\neq G\rbrace.
\label{super}
\end{equation}
  $E_G$ is the endomorphism of $\mathcal{V}(G)$, which extends the character of $\pi$ to the character of its contragredient, defined by: $$E_G(\Theta)= \sum_{1 \leq i \leq k} c_ {i} \Theta _ {\pi_ {i}^{\vee}}$$
where $\pi_ {i}^{\vee}\in\Pi(G)$ is the contragredient of $\pi_ {i}$.
Let $\Theta'\in\mathcal{V}(M)$ then by linearity we define the induction functor:  $$i_{G,M}: \mathcal{V}(M)\longrightarrow\mathcal{V}(G).$$  
We define on the other hand using the constant and weak constant term, the functors of restriction and weak restriction of Jacquet:\\
$$ r_ {M, G}: \mathcal {V} (G) \longrightarrow \mathcal {V} (M), ~ \Theta \longmapsto r_ {M, G} (\Theta)= \Theta_M= \sum_ {1 \leq i \leq k} c_ {i} (\Theta _ {\pi_ {i}})_{M}.
$$ 
$$r^{t}_{M,G}:\mathcal {V}(G)\longrightarrow \mathcal {V}(M),~\Theta \longmapsto r^{t}_{M,G}(\Theta)= \Theta^{w}_{M}= \sum_ {1 \leq i \leq k} c_ {i} {(\Theta _ {\pi_ {i}})}_{M}^{w}.$$\\

Let $ M, L \in \mathcal {L} $, we denote by: $$W ^ {M, L} = \lbrace w \in W_0:~ w (M \cap P_0) \subset P_0,~ w ^ {- 1} (L \cap P_0) \subset P_0\rbrace $$
the subgroup of $ W_0$, such that in each double class $ W_0^L w W_0^M $, there exist a unique element of $ W^{M, L}.$ The following theorem due to I.N. Bernstein and A.V. Zelevinsky [4], gives a good description of the composition of functors $r_{L, G}$ and $i_{G, M}$.\\
\begin{Th}\textsl{ $[4]$~If $ M, L \in \mathcal {L} $, then $r_{L, G} \circ  i_{G, M}$ has a filtration consisting of subfunctors:$$i_{L,L_{w}}\circ w\circ r_{w^{-1}L_{w},M}$$
where $L_w= wM\cap L\subset L,$ and $w\in W ^ {M, L}.$\\}
\end{Th}
~~\\
In particular, in the appropriate space of virtuals characters, if $\Theta'\in\mathcal{V}(M)$ then:\\
\begin{equation} 
 r_{L,G}\circ(i_{G,M}(\Theta'))=\sum_{w\in W^{M,L}} i_{L,L_{w}}(r_{L_{w},wM}(w\Theta')).
 \end{equation}
 \subsection{The duality on \bf{$\mathcal{V}(G)$}}
In [6], K. Betta\"ieb has defined an operator $ D ^ {t} _ {G} $ on the space $ \mathcal {V} (G) $ as follows:
 \begin{equation}
D^{t}_{G}= \sum_{M \in \mathcal {L}}(-1)^{\dim(A_{M})}\hspace{0,2cm} i_{G,M}\circ r^{t}_{M,G}.
\end{equation}
\begin{Th} \textsl{ $[6]$ The operator $D^t_G$ has the following properties:}\begin{enumerate}
\item[1-]\textsl{ If $L\in\mathcal{L}$, then :} $$D^t_G\circ i_{G,L} = i_{G,L}\circ D^t_L~ \textsl{and}~ r^t_{L,G}\circ D^t_G= D^t_L\circ
r^t_{L, G}$$
\item[2-] \textsl{$D^t_G$ is an involution, i.e ${D^t_G}^2= id.$}
\item[3-]\textsl{ If $\pi\in\Pi_2(G),$ then $D^t_G(\Theta_{\pi})= \pm \Theta_\pi$ }
\item[4-]\textsl{$D^t_G$ takes irreducible representations to irreducible representations.}
\item[5-]\textsl{If $M\in\mathcal{L}$, and $\sigma\in\Pi_2(M)$, then $D^t_G$ preserves the irreducible characters of $i_{G, M}(\Theta_\sigma)$ up to sign.}
\end{enumerate}
\end{Th}
~~\\
We can also add some additional properties of $D^t_G$ illustrate in the next lemma.
\begin{Lemma}\textsl{Let $\pi$ be an irreducible tempered representation of $G$}
\begin{enumerate}
\item[1-]\textsl{If $\chi$ is a character of $G$. Then : $$D^t_G(\Theta_{\chi\otimes\pi})= \chi\otimes D^t_G(\Theta_\pi)$$
where $\chi\otimes\pi$ denotes the twist of $\pi$ by $\chi$.}
\item[2-]$D^t_G\circ E_G= E_G\circ D^t_G.$ 
\item[3-]\textsl{Let $\Theta\in\mathcal{V}(G)$, then $D^t_G(\Theta)$ is a finite linear combination of induced of supertempered characters up to sign.}
\end{enumerate}
\end{Lemma}
\begin{proof}\textsl{Let $P= MN$ a standard parabolic subgroup of G and $\sigma$ is a tempered
representation of $M$, then proposition $1.9$ of $[4]$ implies:
$$i_{G, M}(\Theta_{\chi\otimes\sigma})= \chi\otimes i_{G, M}(\Theta_\sigma),~\textsl{and}~ 
 r_{M,G}(\Theta_{\chi\otimes\pi})= \chi\otimes r_{G, M}(\Theta_\pi)$$
 Consequently, $$r^t_{M,G}(\Theta_{\chi\otimes\pi})= \chi\otimes r^t_{G, M}(\Theta_\pi).$$
 The  first assertion now follows from the definition of $D^t_G.$\\
 For the second assertion, given $\pi\in\Pi(G)$ then
according to proposition $3.1.2$ and corollary $4.2.5$ of $[9]$, under the assumption that $\Theta^{w}_M := \Theta^{w}_{\overline{P}}$, is independent of all choices of parabolic subgroups with Levi component $M$ [14], we get:}
\begin{eqnarray*}
E_G\circ D^t_G(\Theta_\pi)&=& E_G\circ \sum_{M \in \mathcal {L}}(-1)^{\dim(A_{M})}\hspace{0,2cm} i_{G,M}\circ r^{t}_{M,G}(\Theta_\pi)\\
&=& \sum_{M \in \mathcal {L}}(-1)^{\dim(A_{M})}\hspace{0.2cm} E_G\circ i_{G,M}\circ r^{t}_{M,G}(\Theta_\pi)\\
&=&\sum_{M \in \mathcal {L}}(-1)^{\dim(A_{M})}\hspace{0.2cm} i_{G,M}\circ E_M\circ r^{t}_{M,G}(\Theta_\pi)\\
&=&\sum_{M \in \mathcal {L}}(-1)^{\dim(A_{M})}\hspace{0.2cm}  i_{G,M}\circ r^{t}_{M,G} \circ E_G(\Theta_\pi)\\
&=& D^t_G\circ E_G(\Theta_\pi).
\end{eqnarray*}
In particular: $$D^t_G\circ E_G(\Theta_{\chi\otimes\pi})= \chi^{-1}\otimes D^t_G(E_G(\Theta_\pi)).$$
 
Let $\Theta\in\mathcal{V}(G)$, then there is an unique finite familly $\lbrace L_i, \Theta_i\rbrace_{1\leq i\leq k}$, where  $L_i\in \mathcal{L}$ and $\Theta_i\in\mathcal{V}_{st}(G)$, such that [5, corollaire 7]: $$\Theta= \sum_{i} i_{G, L_{i}}(\Theta_i)$$ 
so from \eqref{super} and that $D^t_G$ commutes with the induction funtor, we find: $$D^t_G(\Theta)= \sum_{i} i_{G, L_{i}}\circ D^t_{L_i}(\Theta_i)=  \pm\sum_{i} i_{G, L_{i}}(\Theta_i)$$ 
\end{proof}
In the next section, we will define by analogy to [6], a duality on  the space of all characters of irreducible representations in $\Pi(\widetilde{R_\sigma}, \chi_\sigma)$.
\vskip 2em
\maketitle

\large{\section{Duality in $\bf {\mathbb{C} (\Pi (\widetilde{R_\sigma}, \chi_ \sigma))} $}}
We say that $ (M, \sigma) $ is a discrete pair of $ G $ if $ M \in \mathcal {L} $ and $ \sigma \in \Pi_ {2} (M ) $. Let
$ (M, \sigma) $ be a fixed discrete pair of $G$,  $ i_ {G, M} (\sigma) $ the tempered (normalized) representation parabolically induced from $\sigma$ and $ \Pi _ {\sigma} (G) $ the set of all its irreducible constituents.  We write: $$ W_ \sigma = \lbrace w \in W: w \sigma \simeq \sigma \rbrace$$ for the isotropy group  of $\sigma$, and: $$\Delta_\sigma= \lbrace\alpha \in \Phi^{+}(P, A_M): \mu_\alpha(\sigma) = 0\rbrace$$ where $\mu_\alpha(\sigma)$) is the rank one Plancherel measure for $\sigma$ attached
to $\alpha$ [11; p 1108]. To each $w \in W_\sigma, $ there exist an intertwining operator  of the representation $ i_{G, M} (\sigma)$ in itself. Let $\mathcal{C}(\sigma)$ be the commuting
algebra of $i_{G, M} (\sigma)$, and $W^0_\sigma$ the subgroup of $ W_\sigma $, generated by the reflections in the roots of $\Delta_\sigma$, then the $R$-group of $i_{G, M}(\sigma)$, [11],[15]:  $$R _{\sigma}=\lbrace r.\in W_\sigma;~ r\Delta_\sigma= \Delta_\sigma\rbrace\simeq W_\sigma / W^0_\sigma $$  has the property that $\mathcal{C}(\sigma)$ is isomorphic to the complex group algebra $ \mathbb {C} [R_ \sigma] $ twisted by a cocycle $\eta$. As in[1, \S2], let: $$
1\to Z_{\sigma}\to \widetilde{R_\sigma} \to R_{\sigma}\to 1
$$
the central extension of $ R_{\sigma} $, over which the cocycle associated to the intertwining operator split. 
There is a character $ \chi_{\sigma} $ of $ Z_{\sigma} $ such that $ \Pi_{\sigma} (G) $
is parametrize by  $ \Pi (\tilde { R _\sigma}, \chi _ {\sigma})$, the set of irreducible representations $ \rho $ of
$ \widetilde{R_\sigma}$ having $ \chi_{\sigma} $ as $ Z_{\sigma}-$central character.
Let $ \Bbb C (\Pi _ {\sigma} (G))$ (resp $ \Bbb C (\Pi (\widetilde{R_\sigma}, \chi _ {\sigma})) $) the complex vector space generated by
the characters of representations in $ \Pi _ {\sigma} (G) $ (resp. $ \Pi (\widetilde{R_\sigma}, \chi _ {\sigma}) $). Hence the bijection $ \rho \mapsto \pi_ {\rho} $ between $ \Pi (\widetilde{R_\sigma}, \chi _ {\sigma}) $ and $ \Pi _ {\sigma}(G) $ induces an isomorphism:
$$
I ^ {G}: \Bbb C (\Pi (\widetilde{R_\sigma}, \chi _ {\sigma})) \to \Bbb C (\Pi _ {\sigma} (G))
,~ \theta_\rho \mapsto I ^ {G} (\theta_\rho) = \Theta_{\pi_\rho}$$
described in term of intertwining algebra as in [1, p 88].\\

Let $ L \in \mathcal {L} (M) $, we say that $ L $  satisfies the
Compatibility condition of Arthur if $ \mathfrak {a} _ {L} \cap \overline {\mathfrak {a} _ {\sigma}^ {+}} $ contains an open subset of $ \mathfrak {a} _{L}, $ where $\mathfrak {a} _{\sigma}^ {+}:=\lbrace X\in \mathfrak{a}_M;~ \alpha(X)>> 0, \forall\alpha\in\Delta_\sigma\rbrace$, denotes the positive chamber
corresponding to $ \Delta_\sigma$. We denote  by $ \mathcal {L} _{A} (M) $ the set 
 of Levi subgroups $L \in \mathcal {L} (M) $ which satisfies this condition, so if $L\in\mathcal {L} _{A} (M)$, then $R_{\sigma}^{L}:=  R_{\sigma}\cap W^{L}$ is the R-group of $i_{L,M}(\sigma).$ Therefore as at $ G $, we obtain a bijection $ \rho_ {L} \mapsto \tau _ {\rho_ {L}} $ between 
$ \Pi (\widetilde{R^{L}_\sigma}, \chi _ {\sigma}) $ and $ \Pi _ {\sigma} (L) $ which induces again the isomorphism:
 $$ I ^ {L}: \Bbb C (\Pi (\widetilde{R^{L}_\sigma}, \chi _ {\sigma})) \to \Bbb C (\Pi _ {\sigma} (L) ). $$
Notice that the Jacquet-weak restriction $ r ^ {t} _ {L, G} $ does not send $ \Bbb C (\Pi _ {\sigma} (G)) $ into $ \Bbb C (\Pi _ {\sigma} (L)),$ [14, Lemma 3.5], for that we denote by $ r ^ {t} _ {L, G} (\sigma) $ the projection of $ r ^ {t} _{L, G} $ on $ \Bbb C (\Pi_ { \sigma} (L)). $\\
Let $L\in\mathcal {L} _{A} (M)$, we consider the following two functors of induction and restriction :
 $$\textsl{Ind}^{\widetilde{R_\sigma}}_ {\widetilde{R^{L}_\sigma}}: \mathbb{C}(\Pi(\widetilde{R^{L}_\sigma}, \chi_\sigma))\to \mathbb{C}(\Pi(\widetilde{R_\sigma}, \chi_\sigma))$$
and $$\textsl{Res}^{\widetilde{R_\sigma}}_{\widetilde{R^{L}_\sigma}}: \mathbb{C}(\Pi(\widetilde{R_\sigma}, \chi_\sigma))\to \mathbb{C}(\Pi(\widetilde{R^{L}_\sigma}, \chi_\sigma))$$\\
then using $(2.2)$ and the transitivity of induction and restriction functors we get the following commutative diagram [6, Th\'eor\`eme 6]:
\[ \xymatrix{
\Bbb C(\Pi(\widetilde{R^L_\sigma},\chi_{\sigma})) \ar[d]^{I^L} \ar[r]_{\textsl{Ind}^{\widetilde{R_\sigma}}_ {\widetilde{R^{L}_\sigma}}}& \Bbb C(\Pi(\widetilde{R_\sigma},\chi_{\sigma}))\ar[d]^{I^G} \ar[r]_{\textsl{Res}^{\widetilde{R_\sigma}}_{\widetilde{R^{L}_\sigma}}} &\Bbb C(\Pi(\widetilde{R^L_\sigma},\chi_{\sigma})) \ar[d]^{I^L}\\  \Bbb C(\Pi_{\sigma}(L))\ar[r]_{i_{G, L}} & \Bbb C(\Pi_{\sigma}(G))\ar[r]_{r^t_{L, G}(\sigma)} & \Bbb C(\Pi_{\sigma}(L))}
\]

Let $r\in R_\sigma$, and $ \tilde {r}$ its inverse image in $\widetilde{R_\sigma}$, we say that the triplet $(M, \sigma, \tilde{r})$ is a $\widetilde{R_\sigma}-$virtual triplet of $ G$. To each triplet $(M, \sigma, \tilde{r})$,  J. Arthur [1, \S 3] corresponds a distribution-character $\Theta(M, \sigma, \tilde {r}):=  \Theta ^ G (M, \sigma, \tilde {r})$, called the virtual character of $G$ and is written in the following form:\\
\begin{equation}
 \Theta(M, \sigma, \tilde{r}) =\sum_{\rho\in\Pi (\widetilde{R_\sigma}, \chi_\sigma)} \theta_{\rho^\vee}(\tilde{r})~ \Theta_{\pi_\rho}.
\end{equation} 
By inverting $(3.2)$ we will have:
\begin{equation}
\Theta_{\pi_\rho}= {\vert \widetilde{R_\sigma}\vert}^{-1} \sum_{\tilde{r}\in\widetilde{R_\sigma}} \theta_{\rho}(\tilde{r})~\Theta(M, \sigma, \tilde{r}).
\end{equation} 
Those triplets are $W_0-$invariant, and if $z\in Z_\sigma$, then:\\ $$\Theta(z(M, \sigma, \tilde{r}) )=\chi^{-1}_\sigma(z)~\Theta (M, \sigma, \tilde{r})$$\\
So $\Theta(M, \sigma, \tilde{r})$ can be vanished. Following Arthur [1, p 92], the $ \widetilde{R_\sigma}-$triplet $ (M, \sigma, \tilde {r}) $ of $ G $ is said to be essential if, $ \chi_ \sigma\equiv 1$ on $ \lbrace z \in Z_\sigma, z.cl(\tilde {r}) \subset cl(\tilde {r}) \rbrace $, where $ cl (\tilde {r}) $ is the  $ \widetilde{R_\sigma}-$conjugacy
class of $ \tilde {r}. $ Therefore, $ \Theta(M, \sigma, \tilde {r}) $
is non-zero if and only if  the $\widetilde{R_\sigma}-$virtual triplet is essential. The collection of distributions relative to this type of triplets forms a basis of $\mathcal{V}(G)$ [1, \S 3].\\

For each $r\in R_{\sigma}$, we define: $$\mathfrak{a}_{M}^{r}= \lbrace H \in \mathfrak{a}_{M}:\hspace{0,2cm}r.H=H\hspace{0,2cm}\rbrace$$
the space of fixed vectors of $r$.
We know from [7], that there exists a unique $ L_ {r} \in \mathcal{L}(M) $ such that $\mathfrak {a}_{M}^{r} = \mathfrak {a} _ {L_ {r}} $ and that it verifies the compatibility condition of Arthur. We denote by : $$R_{\sigma, \textsl{reg}}=\lbrace r\in R_\sigma;\hspace{0.5cm}\mathfrak{a}_{M}^{r}= \mathfrak{a}_G\rbrace$$
the set of regulars elements in $R_\sigma$, and: $$\mathcal{L}(R_\sigma)= \lbrace S\in\mathcal{L}(M);~ L_r= S,~  \textsl{for some}~ r\in R_\sigma\rbrace$$ so it is an immediate consequence that $$R_\sigma= \bigcup_{S\in\mathcal{L}(R_\sigma)} R^S_{\sigma, \textsl{reg}} $$
the union is disjoint. If $R _ {\sigma, \textsl{reg}}$ is not empty, we say that $R_\sigma$ and the triplet $ (M, \sigma, \tilde {r}),$  $\tilde {r} \in \widetilde{R_{\sigma, \textsl{reg}}}$ the inverse image of $R_{\sigma,\textsl{reg}}$ in $\widetilde{R_\sigma}$ are both elliptic, and so $\Theta(M, \sigma, \tilde{r})$ is a supertempered virtual character of $G$ [14].\\
\begin{Lemma}
\textsl{Let $\rho\in\Pi(\widetilde{R_\sigma}, \chi_\sigma)$, then for every $ L\notin\mathcal {L}_A(M)$:} $(\Theta_{\pi_\rho})^w_L = 0. $
\end{Lemma}
\begin{proof} 
Let $\rho\in\Pi(\widetilde{R_\sigma}, \chi_\sigma)$, and $L \in\mathcal {L}$, then under the asumption that $\pi_\rho$ is tempered, we have:
\begin{align*}
(\Theta_{\pi_\rho})^w_L &= \vert\widetilde{R_\sigma}\vert^{-1}\sum_{\tilde{r}\in\widetilde{R_\sigma}} \theta_\rho (\tilde{r})~ \Theta^G(M, \sigma, \tilde{r})^w_L\\
&= \sum_{S\in\mathcal{L}(R_\sigma)} \vert \widetilde{R^S_{\sigma, \textsl{reg}}}\vert^{-1}\sum_{\tilde{r}\widetilde{R^S_{\sigma, \textsl{reg}}}}\theta_\rho(\tilde{r})~\Theta^G(M, \sigma, \tilde{r})^w_L.
 \end{align*}
For $S\in\mathcal{L}(R_\sigma)$, if $\tilde{r}\in \widetilde{R^S_{\sigma, \textsl{reg}}}$, then $(M, \sigma, \tilde{r})$ is an elliptic virtual triplet of $S$ therefore: 
\begin{align*}
\Theta^G(M, \sigma, \tilde{r})&= \sum_{\tau\in\Pi(S)} \theta_{\rho^\vee_\tau}(\tilde{r})~i_{G, S}(\Theta_\tau)\\
&= i_{G, S}(\Theta^S(M, \sigma, \tilde{r}))
\end{align*}
hence:
\begin{align*}
(\Theta_{\pi_\rho})^w_L &= \sum_{S\in\mathcal{L}(R_\sigma)} \vert \widetilde{R^S_{\sigma, \textsl{reg}}}\vert^{-1}\sum_{\tilde{r}\widetilde{R^S_{\sigma, \textsl{reg}}}}\theta_\rho(\tilde{r})~i_{G, S}(\Theta^S(M, \sigma, \tilde{r}))^w_L\\
&= \sum_{S\in\mathcal{L}(R_\sigma)} \vert \widetilde{R^S_{\sigma, \textsl{reg}}}\vert^{-1}\sum_{\tilde{r}\widetilde{R^S_{\sigma, \textsl{reg}}}}\theta_\rho(\tilde{r})~~ r^t_{L, G}\circ i_{G, S}(\Theta^S(M, \sigma, \tilde{r})).
\end{align*}
By using $(2.2)$ we get: 
$$(\Theta_{\pi_\rho})^w_L= \sum_{S\in\mathcal{L}(R_\sigma)} \vert \widetilde{R^S_{\sigma, \textsl{reg}}}\vert^{-1}\sum_{\tilde{r}\widetilde{R^S_{\sigma, \textsl{reg}}}}\theta_\rho(\tilde{r})\sum_{w\in W^{S, L}} i_{L, L_w}(r^t_{L_w, wS}(w\Theta^S(M, \sigma, \tilde{r})))$$  
where, $L_w= wS\cap L\in\mathcal{L}(L)$. Since $\Theta^S(M, \sigma, \tilde{r})\in\mathcal{V}_{st}(S)$, we have $w\Theta^S(M, \sigma, \tilde{r})\in\mathcal{V}_{st}(wS)$, hence $r^t_{L_w, wS}(w\Theta^S(M, \sigma, \tilde{r}))= 0$, unless there is $S\in \mathcal{L}(R_\sigma)$, and $w\in W^{S, L}$ such that, $L_w= wS$.\\
So, $r^t_{L_w, wS}(w\Theta^S(M, \sigma, \tilde{r}))= 0$, unless there is $w\in W^{S, L}$ such that, $S\subset w^{-1}L$, if this is the case, we get: $$w^{-1}\mathfrak{a}_L= \mathfrak{a}_{w^{-1}L}\subset \mathfrak{a}_S= \mathfrak{a}^r_M= \mathfrak{a}_{L_r}\subset \mathfrak{a}_M$$
for some $r\in R_\sigma$. Therefore without loss of generality, and since $L_r\in\mathcal{L}_A(M)$, we assume that $L\in\mathcal{L}_A(M)$, [1, p 90].\\
 \end{proof}
 
We define formally the sign application by: 
$$ \xi_{\widetilde{R_\sigma}}:\hspace{0,2cm} \widetilde{R_\sigma} \to \{ \pm 1 \}: \xi_{\widetilde{R_\sigma}}(\tilde{r}) = (-1)^{\dim(A_{L_{r}})}.$$
By analogy, to $ D ^{t}_{G}$ we define an operator $ D _ {\widetilde {R_\sigma}, \chi_{\sigma}} $ on $ \Bbb C ( \Pi (\widetilde {R_\sigma}, \chi _ {\sigma})) $ by:
\begin{equation}
D _ {\widetilde {R_\sigma}, \chi_{\sigma}}=\sum_{L\in \mathcal {L}_{A}(M)} (-1)^{\dim (A_{L})}\hspace{0,2cm} \textsl{Ind}^{\widetilde{R_\sigma}}_{\widetilde{R^{L}_\sigma}}\circ
\textsl{Res}^{\widetilde{R_\sigma}}_{\widetilde{R^{L}_\sigma}}
\end{equation}
If the cocycle splits, then the operator is simply $D_{R_\sigma}$.
The next theorem presents some properties of this operator.
\begin{Th}\textsl{The operator $D _ {\widetilde {R_\sigma}, \chi_{\sigma}}$ has the following properties:}
\begin{enumerate}
\item[1-]\textsl{Let $L \in \mathcal {L}_A(M)$, then the folowing diagram is commutative:}
\[ \xymatrix{
\Bbb C(\Pi(\widetilde{R_\sigma},\chi_{\sigma})) \ar[r]^{\textsl{Res}^{\widetilde{R_\sigma}}_{\widetilde{R^{L}_\sigma}}}\ar[d]_{D _ {\widetilde {R_\sigma}, \chi_{\sigma}}} & \Bbb C(\Pi(\widetilde{R^L_\sigma},\chi_{\sigma}))\ar[r]^{\textsl{Ind}^{\widetilde{R_\sigma}}_ {\widetilde{R^{L}_\sigma}}}  \ar[d]_{D _ {\widetilde {R^L_\sigma}, \chi_{\sigma}}} &\Bbb C(\Pi(\widetilde{R_\sigma},\chi_{\sigma}))\ar[r]^{I^G} \ar[d]_{D _ {\widetilde {R_\sigma}, \chi_{\sigma}}} & \Bbb C(\Pi_{\sigma}(G))\ar[d]_{D^t_G}\\
 \Bbb C(\Pi(\widetilde{R_\sigma},\chi_{\sigma})) \ar[r]^{\textsl{Res}^{\widetilde{R_\sigma}}_{\widetilde{R^{L}_\sigma}}} & \Bbb C(\Pi(\widetilde{R^L_\sigma},\chi_{\sigma}))\ar[r]^{{\textsl{Ind}^{\widetilde{R_\sigma}}_ {\widetilde{R^{L}_\sigma}}}} &\Bbb C(\Pi(\widetilde{R_\sigma},\chi_{\sigma}))\ar[r]^{I^G} & \Bbb C(\Pi_{\sigma}(G))} 
\] 
\item[2-]$D _ {\widetilde {R_\sigma}, \chi_{\sigma}}$ \textsl{is an involution.}
 \end{enumerate}
\end{Th}
\begin{proof}~~
\begin{enumerate}
\item[1-] \textsl{Let $L \in \mathcal {L}_A(M)$, according to the last diagram :}
$$r^{t}_{L,G}({\sigma})\circ I^{G}= I^{L} \circ  \textsl{Res}^{\widetilde{R_\sigma}}_{\widetilde{R^{L}_\sigma}}.$$
\textsl{The composition on both sides of this last expression with the induction functor $ i_ {G, L}, $ gives us:
 $$i_{G,L} \circ r^{t}_{L,G}({\sigma})\circ I^{G}=i_{G,L}\circ I^{L} \circ  \textsl{Res}^{\widetilde{R_\sigma}}_{\widetilde{R^{L}_\sigma}}$$
Since $$i_{G,L}\circ I^{L}=I^{G} \circ  \textsl{Ind}^{\widetilde{R_\sigma}}_{\widetilde{R^{L}_\sigma}}$$
then $$i_{G,L} \circ r^{t}_{L,G}({\sigma})\circ I^{G}=I^{G} \circ  \textsl{Ind}^{\widetilde{R_\sigma}}_{\widetilde{R^{L}_\sigma}}\circ  \textsl{Res}^{\widetilde{R_\sigma}}_{\widetilde{R^{L}_\sigma}}$$
If we apply the sum on $\mathcal {L}$, we obtain:
$$
 \sum_{L\in\mathcal {L}}(-1)^{\dim(A_L)}\hspace{0.2cm}  i_{G,L}\circ r^{t}_{L,G}({\sigma}) \circ I^{G}=$$ $$I^{G} \circ  \sum_{L\in\mathcal {L}}(-1)^{\dim(A_L)}\hspace{0.2cm}  \textsl{Ind}^{\widetilde{R_\sigma}}_{\widetilde{R^{L}_\sigma}}\circ  \textsl{Res}^{\widetilde{R_\sigma}}_{\widetilde{R^{L}_\sigma}}$$ 
and since $\mathcal {L_{A}}(M)\subset\mathcal {L}$, then by decomposing the sum we get: 
$$\hspace{1.5cm} D^t_G\circ I^G= I^G\circ D _ {\widetilde {R_\sigma}, \chi_{\sigma}} + \left[\sum_{L\notin\mathcal {L}_A(M)}(-1)^{\dim(A_L)+1}\hspace{0.1cm}  i_{G,L}\circ r^{t}_{L,G}({\sigma})\right] \circ I^G$$
by lemma $3.1$: $\sum\limits_{L\notin\mathcal {L}_A(M)}(-1)^{\dim(A_L)+1}\hspace{0.1cm}  i_{G,L}\circ r^{t}_{L,G}({\sigma})$ vanishes, hence:} $$D^t_G\circ I^G= I^G\circ D _ {\widetilde {R_\sigma}, \chi_{\sigma}}.$$ 
\textsl{Let now $L \in \mathcal {L}_{A}(M),$ then  following the last result, we have on the one hand:
$$
\hspace{1.5cm} I^{L} \circ \textsl{Res}^{\widetilde{R_\sigma}}_{\widetilde{R^{L}_\sigma}}\circ D _ {\widetilde {R_\sigma}, \chi_{\sigma}} = r^{t}_{L,G}({\sigma})\circ I^{G} \circ D _ {\widetilde {R_\sigma}, \chi_{\sigma}} = r^{t}_{L,G}({\sigma})\circ D^{t}_{G} \circ I^{G}
$$
and on the other hand, following theorem $2.2.1$, and again the last result but applied to $ L $, we get:}
$$
\hspace{2cm} I^{L} \circ D _ {\widetilde {R^L_\sigma}, \chi_{\sigma}}\circ \textsl{Res}^{\widetilde{R_\sigma}}_{\widetilde{R^{L}_\sigma}} = D^{t}_{L} \circ I^{L}\circ \textsl{Res}^{\widetilde{R_\sigma}}_{\widetilde{R^{L}_\sigma}} = D^{t}_{L} \circ r^{t}_{L,G}({\sigma})\circ  I^{G}$$ $$\hspace{1.8cm}= r^{t}_{L,G}({\sigma})\circ D^{t}_{G} \circ I^{G}$$
which implies: $$\hspace{2cm} I^{L} \circ \textsl{Res}^{\widetilde{R_\sigma}}_{\widetilde{R^{L}_\sigma}}\circ D _ {\widetilde {R_\sigma}, \chi_{\sigma}}  = I^{L} \circ D _ {\widetilde {R^L_\sigma}, \chi_{\sigma}}\circ \textsl{Res}^{\widetilde{R_\sigma}}_{\widetilde{R^{L}_\sigma}}.$$
Since the operator $ I ^ {L} $ is an isomorphism for all $ L \in \mathcal {L_ {A}} (M) $, we will have: $$\textsl{Res}^{\widetilde{R_\sigma}}_{\widetilde{R^{L}_\sigma}}\circ D _ {\widetilde {R_\sigma}, \chi_{\sigma}}= D _ {\widetilde {R^L_\sigma}, \chi_{\sigma}}\circ \textsl{Res}^{\widetilde{R_\sigma}}_{\widetilde{R^{L}_\sigma}}.$$
In the same way we have:
\begin{align*}
I^G\circ \textsl{Ind}^{\widetilde{R_\sigma}}_{\widetilde{R^{L}_\sigma}}\circ D _ {\widetilde {R^L_\sigma}, \chi_{\sigma}}&= i_{G, L}\circ I^L\circ D _ {\widetilde {R^L_\sigma}, \chi_{\sigma}}\\
&=i_{G, L}\circ D^{t}_L\circ I^L\\
&= D^{t}_G\circ i_{G, L}\circ I^L\\
&= D^{t}_G\circ I^G\circ \textsl{Ind}^{\widetilde{R_\sigma}}_{\widetilde{R^{L}_\sigma}}\\
&= I^G\circ D _ {\widetilde {R_\sigma}, \chi_{\sigma}}\circ \textsl{Ind}^{\widetilde{R_\sigma}}_{\widetilde{R^{L}_\sigma}}.
\end{align*}
\item[2-]{By definition of the operator $D _ {\widetilde {R_\sigma}, \chi_{\sigma}}$, and the transitivity of the induction and restriction functors, we have:
\begin{align*}
D^2_{\widetilde{R_\sigma}}&= D _ {\widetilde {R_\sigma}, \chi_{\sigma}}\circ\left[ \sum_{L\in \mathcal {L}_{A}(M)}(-1)^{\dim(A_L)}~ \textsl{Ind}^{\widetilde{R_\sigma}}_{\widetilde{R^{L}_\sigma}}\circ \textsl{Res}^{\widetilde{R_\sigma}}_{\widetilde{R^{L}_\sigma}}\right]\\
&=\sum_{L\in \mathcal {L}_{A}(M)}(-1)^{\dim(A_L)} D _ {\widetilde {R_\sigma}, \chi_{\sigma}}\circ \textsl{Ind}^{\widetilde{R_\sigma}}_{\widetilde{R^{L}_\sigma}}\circ \textsl{Res}^{\widetilde{R_\sigma}}_{\widetilde{R^{L}_\sigma}}\\
&=\sum_{L\in \mathcal {L}_{A}(M)}(-1)^{\dim(A_L)}~ \textsl{Ind}^{\widetilde{R_\sigma}}_{\widetilde{R^{L}_\sigma}}\circ D _ {\widetilde {R^L_\sigma}, \chi_{\sigma}}\circ \textsl{Res}^{\widetilde{R_\sigma}}_{\widetilde{R^{L}_\sigma}}\\
&=\sum_{L\in \mathcal {L}_{A}(M)}(-1)^{\dim(A_L)}~ \textsl{Ind}^{\widetilde{R_\sigma}}_{\widetilde{R^{L}_\sigma}}\circ\sum_{K\in\mathcal{L}^{L}_A(M)}(-1)^{\dim(A_K)}~ \textsl{Ind}^{\widetilde{R^L_\sigma}}_{\widetilde{R^K_\sigma}}\circ \textsl{Res}^{\widetilde{R^L_\sigma}}_{\widetilde{R^K_\sigma}}\circ \textsl{Res}^{\widetilde{R_\sigma}}_{\widetilde{R^{L}_\sigma}}\\
&=\sum_{L\in \mathcal {L}_{A}(M)}(-1)^{\dim(A_L)} \sum_{K\in\mathcal{L}^{L}_A(M)}(-1)^{\dim(A_K)}~ \textsl{Ind}^{\widetilde{R_\sigma}}_{\widetilde{R^K_\sigma}}\circ \textsl{Res}^{\widetilde{R_\sigma}}_{\widetilde{R^K_\sigma}}\\
&=\sum_{K\in \mathcal {L}_{A}(M)} \sum_{L\in\mathcal{L}_A(K)}(-1)^{\dim(A_K)- \dim(A_L)}~ \textsl{Ind}^{\widetilde{R_\sigma}}_{\widetilde{R^K_\sigma}}\circ \textsl{Res}^{\widetilde{R_\sigma}}_{\widetilde{R^K_\sigma}}.
\end{align*}
Since: $\sum\limits_{L\in\mathcal{L}_A(K)}(-1)^{\dim(A_K)- \dim(A_L)}=0$ if $K\neq G$, $[16,\textsl{Lemme}~ 1.2.3 ],$ then: $D _ {\widetilde {R_\sigma}, \chi_{\sigma}}^2= Id.$}
\end{enumerate}
\end{proof}
\begin{Prop}\textsl{Let $(M, \sigma)$ be a discrete pair of $G$}
\begin{enumerate}
\item[1-] \textsl{For all $\rho\in \Pi (\widetilde{R_\sigma}, \chi_\sigma), D _ {\widetilde {R_\sigma}, \chi_{\sigma}}(\theta_\rho)$ is irreducible.}
\item[2-]\textsl{The dual in ${\widetilde{R_\sigma}}$ of the trivial irreducible character, $1_{\widetilde{R_\sigma}}\in\Bbb C(\Pi(\widetilde{R_\sigma},\chi_{\sigma})),$ is equal to the sign application $\xi_{\widetilde{R_\sigma}}.$}
 \end{enumerate}
\end{Prop}

\begin{proof}
\textsl{Let $\theta_\rho\in\mathbb{C}(\Pi(\tilde{\Re}_{\sigma}, \chi_\sigma)$, according to the last theorem we have : $$I^G(D _ {\widetilde {R_\sigma}, \chi_{\sigma}}(\theta_\rho))= D^t_G(I^G(\theta_\rho))$$ since the operator $ D ^ t_G $ preserves up to isomorphism, the irreducibility of the characters irreducible, and $I^G$ is an isomorphism, it follow that $D _ {\widetilde {R_\sigma}, \chi_{\sigma}}(\theta_\rho)$ is also irreducible. For the second assertion, we have by definition:
$$
 D _ {\widetilde {R_\sigma}, \chi_{\sigma}}(1_{\widetilde{R_\sigma}}(\tilde{r})) = \sum_{L\in \mathcal {L}
_{A}(M)}(-1)^{dim(A_{L})}\hspace{0,2cm}\textsl{Ind}^{\widetilde{R_\sigma}}_{\widetilde{R^{L}_\sigma}}(1_{\widetilde{R^L_\sigma}}(\tilde{r})), \hspace{0.5cm}\tilde{r}\in \widetilde{R_\sigma}
$$
notice that if:~ $\mathfrak{a}_{M}^{r}=\mathfrak{a}_{G}$ for a certain $r \in R_{\sigma}$ then  $\tilde{r}\in \widetilde{R_\sigma}$ is not conjugated to any element of $\widetilde{R^L_\sigma},\hspace{0,2cm}L \in\mathcal {L}_{A}(M)$ therefore $\textsl{Ind}^{\widetilde{R_\sigma}}_{\widetilde{R^{L}_\sigma}} (1_{\widetilde{R^L_\sigma}}(\tilde{r})) = 0$, hence:
$$(\ast)\hspace{3cm} D _ {\widetilde {R_\sigma}, \chi_{\sigma}}(1_{\widetilde{R_\sigma}}(\tilde{r}))= (-1)^{dim(A_{G})}\hspace{1cm}\textsl{si}\hspace {0.2cm}\mathfrak{a}_{M}^{r}=\mathfrak{a}_{G}.\hspace{2.4cm}$$
Now if  $r \in R_{\sigma}$ such that $\mathfrak{a}_{M}^{r}=\mathfrak{a}_{L_{r}}$ then its inverse image $\tilde{r} \in \widetilde{R^{L_{r}}}_{\sigma}$, so:
$$
D_ {\widetilde {R_\sigma}, \chi_{\sigma}}(1_{\widetilde{R_\sigma}}(\tilde{r})) =  \textsl{Res}^{\widetilde{R_\sigma}}_{\widetilde{R^{L_{r}}_\sigma}} (D _ {\widetilde {R_\sigma}, \chi_{\sigma}}(1_{\widetilde{R_\sigma}}(\tilde{r})) 
= D_{\widetilde{R^{L_{r}}_\sigma}} ( \textsl{Res}^{\widetilde{R_\sigma}}_{\widetilde{R^{L_{r}}_\sigma}}(1_{\widetilde{R_\sigma}}(\tilde{r}))$$ $$\hspace{1.4cm} = D_{\widetilde{R^{L_{r}}_\sigma}} (1_{\widetilde{R^{L_{r}}_\sigma}}(\tilde{r})) = (-1)^{\dim(A_{L_{r}})}= \xi_{\widetilde{R_\sigma}}(\tilde{r})
$$
the last equality is due to the expression $(\ast)$ applied to $L_{r}$ instead of $G.$}
\end{proof}
\vskip 2em
\maketitle

\large{\section{The character of Steinberg}}
The purpose of this section is to define a  special character, belonging to
 $\mathcal {V}(G)$. After that, we will see some of its remarkable properties.
Let $(M, \sigma)$ be a discrete pair of $G$, and  $L \in \mathcal {L}_{A}(M)$, for a representation  $\tau_{1}\in \Pi_ {\sigma} (L)$ corresponding to the trivial representation
$1_{\widetilde{R^L_\sigma}}
\in\Pi(\widetilde{R^L_\sigma},\chi_{\sigma})$, we call character of
Steinberg associated to $(M,\sigma)$ the character:
\begin{equation}
 St^{G}_{(M,\sigma)}= \sum_{L\in \mathcal {L}_{A}(M)} (-1)^{\dim(A_{L})}\hspace{0,2cm} i_{G,L}(\Theta_{\tau_{1}}).
\end{equation}
\begin{Th}\textsl {Let $ (M, \sigma) $ be a discrete pair of $ G,$ and $\pi_1 \in \Pi_{\sigma} (G) $ corresponding to the trivial representation, then the character of $\pi_1$ is  equal to the dual of  $St^{G}_{(M, \sigma)}$. In particular $St^{G}_{(M, \sigma)}$ is an irreducible character of $i_{G, M}(\Theta_{\sigma}).$}
\end{Th}
\begin{proof}\textsl{ Let $\Theta_{\pi_1}$ be the character of $\pi_1 \in \Pi_{\sigma}(G)$ corresponding to the trivial character $1_{\widetilde{R_\sigma}}$, then: $\Theta_{\pi_1} =  I^{G} (1_{\widetilde{R_\sigma}})$.\\
According to  the previous proposition and theorem $3.2.1$, we have:}
$$
D^{t}_{G}(\Theta_{\pi_1})= D^{t}_{G}\circ I^{G}(1_{\widetilde{R_\sigma}}) = I^{G} ( D _ {\widetilde {R_\sigma}, \chi_{\sigma}}(1_{\widetilde{R_\sigma}})) = I^{G} (\xi_{\widetilde{R_\sigma}} )
$$
\textsl{We have also:}
\begin{eqnarray*}
D^{t}_{G}(\Theta_{\pi_1}) = I^{G}\circ D _ {\widetilde {R_\sigma}, \chi_{\sigma}}(1_{\widetilde{R_\sigma}}) &=& I^{G}\circ  \sum_{L\in \mathcal {L}_{A}(M)}(-1)^{\dim(A_{L})}\hspace{0.1cm} \textsl{Ind}^{\widetilde{R_\sigma}}_{\widetilde{R^{L}_\sigma}} \circ \textsl{Res}^{\widetilde{R_\sigma}}_{\widetilde{R^{L}_\sigma}} (1_{\widetilde{R_\sigma}}) \\
&=&\sum_{L\in \mathcal {L}_{A}(M)}(-1)^{\dim(A_{L})}~  i_{G,L} \circ I^{L}(1_{\widetilde{R^L_\sigma}})\\ 
&=& \sum_{L\in \mathcal {L}_{A}(M)}(-1)^{dim(A_{L})}~ i_{G,L} (\Theta_{\tau_{1}})\\
&=& St^{G}_{(M,\sigma)} .
\end{eqnarray*}
\textsl{It follows that  $St^{G}_{(M,\sigma)}$ is an irreducible character of $i_{G,M}(\Theta_{\sigma}).$}\\
\end{proof}
\begin{Rem} \textsl{We have just found, that we have:
$$
St^{G}_{(M,\sigma)} = I^{G} (\xi_{\tilde{\Re}_{\sigma}} )
$$
which implies that $St_{(M,\sigma)}^{G}$ corresponding to $\xi_{\widetilde{R_\sigma}}.$ Therefore  $\xi_{\tilde{\Re}_{\sigma}}$ is an irreducible character, it is
defined as a sign character of the group $\widetilde{R_\sigma}.$}
\end{Rem}
\begin{Cor} \textsl{Let $(M,\sigma)$ be a discrete pair of $G$ and $L\in \mathcal {L}_{A}(M)$ then:}
$$r^{t}_{L,G}(\sigma)(St^{G}_{(M,\sigma)}) = St_{(M,\sigma)}^{L}.$$
\end{Cor}
\begin{proof} \textsl{It is clear that:
$$\bigstar\hspace{4cm}
St^{G}_{(M,\sigma)} = I^{G}(  D _ {\widetilde {R_\sigma}, \chi_{\sigma}}(1_{\widetilde{R_\sigma}})).\hspace{4.2cm}
$$
The composition with the projection $ r ^ {t}_{L, G}(\sigma)$, gives us:$$
r^{t}_{L,G}(\sigma)(St^{G}_{(M,\sigma)}) = r^{t}_{L,G}(\sigma)\circ I^{G} \circ D _ {\widetilde {R_\sigma}, \chi_{\sigma}}(1_{\widetilde{R_\sigma}} )=  I^{L}\circ \textsl{Res}^{\widetilde{R_\sigma}}_{\widetilde{R^{L}_\sigma}}  \circ D _ {\widetilde {R_\sigma}, \chi_{\sigma}}(1_{\widetilde{R_\sigma}})
$$
$$\hspace{3.5cm}=  I^{L}\circ D _ {\widetilde {R^L_\sigma}, \chi_{\sigma}}( \textsl{Res}^{\widetilde{R_\sigma}}_{\widetilde{R^{L}_\sigma}} (1_{\widetilde{R_\sigma}}))=  I^{L}(D _ {\widetilde {R^L_\sigma}, \chi_{\sigma}}(1_{\widetilde{R^L_\sigma}}))=  St_{(M,\sigma)}^{L} 
$$
where the last equality is due to the expression $\bigstar$ applied to $ L $.}
\end{proof}

\begin{Cor} \textsl{Under the hypotheses of the theorem $4.1$, we have :
 $$(St_{(M,\sigma)}^{G})^{e} =\pm (\Theta_{\pi_1})^{e}.$$}
\end{Cor}
\begin{proof}
\textsl{We have $D^{t}_{G}(\Theta_{\pi_1}) = St_{(M,\sigma)}^{G}$, then by restriction to $G_{ell}$, we get:
$$
(St_{(M,\sigma)}^{G})^{e} = (D^{t}_{G}(\Theta_{\pi_1}))^{e} = \left( \sum_{L\in\mathcal{L}(M)}(-1)^{\dim(A_L)}~i_{G, L}\circ r^t_{L, G}(\Theta_{\pi_1})\right)^{e}
.$$
Since, the restriction of a character properly induced to $ G_{ell} $ is zero, then 
$$(D^{t}_{G}(\Theta_{\pi_1}))^{e}= \pm (\Theta_{\pi_1})^{e}$$}
\end{proof}

\begin{Prop} \textsl{Let $ (M, \sigma) $ be a discrete pair of $ G$. Then
 $St_{(M,\sigma)}^{G}$ is elliptic if and only if  $R_{\sigma}$ is elliptic.}
\end{Prop}
\begin{proof}\textsl{Since $St_{(M,\sigma)}^{G}$ is an irreducible character of $i_{G,M}(\Theta_{\sigma})$ and corresponding to $\hspace {0,2cm} \xi_{\widetilde{R_\sigma}}$, then  from $[1; \textsl{Proposition}~ 2.1]$, $St_{(M,\sigma)}^{G}$ is elliptic if and only if the restriction of $\xi_{\widetilde{R_\sigma}}$ to $\widetilde{R_{\sigma,\textsl{reg}}}$ is non-zero, or for all $\tilde{r}\in\widetilde{R_{\sigma, \textsl{reg}}}$: $$\xi_{\widetilde{R_\sigma}}(\tilde{r})=  (-1)^{\dim(A_{G})},$$ hence the result.}
\end{proof}
\vskip 2em
 We end this note with a direct application. 
 We know that the choice of a minimum parabolic subgroup $ P_{0} $ of $ G $ determines a base $ \Delta$ of the set of reduced roots $\Phi(G, A_0)$. If $ I \subset \Delta $, let $ M_{I} $
the Levi subgroup of $G$ defined by $I.$ \hspace {0,1cm} We have the following proposition.

\begin{Prop}\textsl{Let $ (M_I, \sigma) $ be a discrete pair of $ G$, where $\hspace {0,2cm} I\subset \Delta.$
If  $I = \Delta - \{\alpha\}$, for a simple root $\alpha$ and $i_{G,M_{I}}(\Theta_{\sigma})$ is reducible, then $\Theta_{\pi_1}$ and  $\pm St_{(M_{I},\sigma)}^{G}$ are the only irreducible characters
 of $i_{G,M_{I}}(\Theta_{\sigma}).$\hspace{0,1cm} Moreover these two irreducible characters are elliptic.}
\end{Prop}
\begin{proof}\textsl{If $i_{G,M_{I}}(\Theta_{\sigma})$ is reductible, then the last theorem and the  decomposition of $St_{(M_{I},\sigma)}^{G}$ show that $\Theta_{\pi_1}$ and $\pm St_{(M_{I},\sigma)}^{G}$ are the only irreducibles characters of $i_{G,M_{I}}(\Theta_{\sigma}).$\hspace{0,1cm}On the other hand, if $i_{G,M_{I}}(\Theta_{\sigma})$ is reductible then the reflexion 
 $s_{\alpha} \in R_{\sigma}$, it follow from corollaire $4.4$ and proposition $4.5$ that the 
 irreducibles characters $\Theta_{\pi_1}$ and $\pm St_{(M_{I},\sigma)}^{G}$ are elliptics because $R_{\sigma,\textsl{reg}}\neq \varnothing$ ($s_{\alpha} \in R_{\sigma,\textsl{reg}}$).}
\end{proof}

\begin{Rem}\textsl{We fix a discrete pair $(M,\sigma)$ of $G.$\hspace{0,1cm}We saw that $St_{(M,\sigma)}^{G}$ is an irreducible character of $i_{G,M}(\Theta_{\sigma})$ and it is  elliptic if and only if  $R_{\sigma}$ is also elliptic. With the proposition $4.6$, we want to extend this result to all the characters of $i_{G,M}(\Theta_{\sigma})$ and thus get out of the choices imposed by J.Arthur ( fixed central extension over which the cocycle $\chi_{\sigma}$ splits, essential virtual triplet, positive chamber, ...), as in $[5]$ for the classification of irreducible, tempered representations (of reductive p-adic groups).}
\end{Rem}
\vskip 2em
\maketitle

\end{document}